\documentclass[a4paper,11pt,english]{amsart}

\usepackage{amsfonts}
\usepackage{amsmath,amsthm}
\usepackage{latexsym}
\usepackage{array}
\usepackage{amssymb}
\usepackage{graphics}
\usepackage[english]{babel}
\usepackage{enumitem}

\usepackage{color}
\definecolor{marin}{rgb}   {0.,   0.3,   0.7} 
\definecolor{rouge}{rgb}   {0.8,   0.,   0.} 
\definecolor{sepia}{rgb}   {0.8,   0.5,   0.} 
\usepackage[colorlinks,citecolor=marin,linkcolor=rouge,
            bookmarksopen,
            bookmarksnumbered
           ]{hyperref}

\theoremstyle{plain} 
\newtheorem{theorem}{Theorem}[section]
\newtheorem{lemma}[theorem]{Lemma}
\newtheorem{proposition}[theorem]{Proposition}
\newtheorem{corollary}[theorem]{Corollary} 
\newtheorem{definition}[theorem]{Definition} \theoremstyle{remark}
\newtheorem{remark}[theorem]{Remark}

\newcommand{\R}{  \mathbb{R}   }

\newcommand{\C}{  \mathbb{C}   }
\newcommand{\Z}{  \mathbb{Z}   }
\newcommand{\N}{  \mathbb{N}   }

\newcommand{\Hc}{  \mathcal{H}   }

\newcommand{\Sc}{  \mathcal{S}   }

\newcommand{\dd}{  \text{d}   }

\renewcommand{\l}{  \ell  }

\renewcommand{\phi}{  \varphi  }

\newcommand{\be}{\begin{equation}}
\newcommand{\ee}{\end{equation}}
\newcommand{\ben}{\begin{equation*}}
\newcommand{\een}{\end{equation*}}

\newcommand{\Norm}[2]{\|#1\|\left.\vphantom{T_{j_0}^0}\!\!\right._{#2}}         
 
\newcommand{\scal}[2]{\langle #1,#2 \rangle}    

\numberwithin{equation}{section}
 \newcommand{\Hcal}{\mathcal{H}}
\author{Erwan Faou}
\address{INRIA \& IRMAR. 
ENS Rennes, Avenue Robert Schumann F-35170 Bruz, France. } 
\email{Erwan.Faou@inria.fr}

 \author{Tiphaine J\'ez\'equel}
\address{IUT de Lannion, Universit\'e de Rennes 1. Rue Edouard Branly, BP 30219, 22302 Lannion Cedex, France.  } 
\email{Tiphaine.Jezequel@univ-rennes1.fr}
\newcommand{\Hct}{\widetilde{\Hc}}

\title[Normalized gradient for computing ground states]
{Convergence of a normalized gradient algorithm for computing ground states}

\begin{document}

\begin{abstract}
We consider the approximation of the ground state of the one-dimensional cubic nonlinear Schr\"odinger equation by a normalized gradient algorithm combined with linearly implicit time integrator, and finite difference space approximation. 
We show that this method, also called {\em imaginary time evolution method} in the physics literature, is convergent, and we provide error estimates: the algorithm converges exponentially towards a modified solitons that is a space discretization of the exact soliton, with error estimates depending on the discretization parameters. \end{abstract}                                                                                                

\subjclass{ 35Q1, 35Q55, 65N12, 65J15}
\keywords{Nonlinear Schr\"odinger equation, ground state, imaginary time method}
\thanks{This work has been supported by the ERC Starting Grant GEOPARDI No 279389
}

\maketitle
%\tableofcontents

\section{Introduction}

The goal of this paper is to give a convergence proof of a normalized gradient algorithm used to compute numerically ground states of Schr\"odinger equations fulfilling symmetry and coercivity conditions 
as considered in the seminal works of Weinstein \cite{Weinstein85} and Grillakis, Shatah and Strauss \cite{Grill87,Grill90}. 
This algorithm is also called {\em imaginary time method} Êin the physics literature: see for instance
\cite{EB95,Adh1,Adh2,CST00,baotang} and the reference therein. 
Let us describe the algorithm in the case of the
focusing cubic non linear Schr\"odinger equation
\begin{equation}\tag{NLS}\label{NLS}
i\partial_t\psi=-\frac12\Delta\psi-|\psi|^2\psi, 
\end{equation}
set on $\R$, 
where $\psi(t,x)$ depends on space variables $x \in \R$. With this equation is associated the energy \begin{equation}\label{H}
H(\psi,\bar\psi)= \frac14\int_\R |\nabla\psi|^2 -|\psi|^4,
\end{equation}
that is preserved by the flow of (NLS) for all times. 
The equation (NLS) can be written 
$$i\partial_t\psi=-\frac12\Delta\psi-|\psi|^2\psi=2 \frac{\partial  H}{\partial \bar \psi}(\psi,\bar \psi).$$

In the rest of this paper, the notation $\nabla H$ will denote the $L^2$ derivative of the energy $H$ with respect to {\em real} functions $\psi$. Note that we have for a real function $u$

$$
\nabla H(u) = 2 \frac{\partial H}{\partial \bar \psi}(u,u) = - \frac{1}{2}\Delta u -  u^3,
$$
the left-hand side denoting the Fr\'echet derivative of $H(u)$ considered as a functional acting on real functions. Note that naturally, $\nabla H \in H^{-1}$ with the embedding $H^{-1} \subset L^2 \subset H^1$.

With these notations, the ground state $\eta(x)$ is defined as the unique real symmetric minimizer in $H^1$ of the problem 
\begin{equation}
\label{minibar}
\min\limits_{\Norm{\psi}{L^2}=1}H(\psi).
\end{equation}
In the one dimensional cubic case considered in this paper, 
explicit computations show that 
\begin{equation*}
\eta(x):=\frac{1}{2} \text{sech}\left(\frac{x}2\right).
\end{equation*}
We denote by $\lambda$ the Lagrange multiplier associated with this minimization problem, such that 
\begin{equation}
\label{etacube}
\nabla H(\eta)=- \frac12 \Delta \eta - \eta^3 = -\lambda\eta.
\end{equation}
With the ground state $\eta$ is associated the solution $\eta(t,x) = \eta(x) e^{i \lambda t}$ of the time dependent equation \eqref{NLS}. By using translation and scaling, the ground state gives rise to a family of explicit solutions of \eqref{NLS} that have the property to be {\em orbitally stabe} in $H^1$, see \cite{Weinstein85,Grill87,Grill90,BFG}. For instance, any solution starting close to $\eta(x)$ will remain close to the manifold $\{ e^{i\alpha} \eta ( x - c)\, | \alpha,c \in \R\} \subset H^1$, the rotation and translation being natural invariant group actions of the nonlinear Schr\"odinger equation \eqref{NLS}.

In more general situations, $\eta$ is not explicitly known, and one has to rely on numerical simulations to compute it. 
To this aim, the {\em imaginary time method}, which is a nonlinear version of the normalized gradient algorithm is widely used. The goal of the present paper is to analyse the efficiency of this method in the simple case described above (the results obtained are in fact valid in more general situations - essentially all situations where the Grillakis-Shatah-Strauss arguments apply). 

The time-discretized algorithm consists in defining a sequence a functions $\{\psi_n\}_{n \in \N}$ as follows: 
\begin{itemize}
\item[(i)]  An intermediate function $\psi_{n}^*$ is defined as a numerical approximation of the solution of the parabolic equation 
\begin{equation}
\label{telephone}
\partial_t \psi = \frac12\Delta\psi + |\psi|^2\psi = - \nabla H(\psi)
\end{equation}
over a time interval $[0,\tau]$, where $\tau$ is a given time step. 
To compute $\psi_n^*$, we will see that the best results are given by the linearly implicit method 
\begin{equation}
\label{poulguen}
\psi_n^*  = \psi_n -  \tau \widehat{ \nabla H}(\psi_n,\psi_n^*), 
\end{equation}
where 
$$
 -  \widehat {\nabla H}(\psi_n,\psi_n^*) =  \frac12\Delta\psi_n^* + 
 |\psi_n|^2\psi_n^{*};
 $$

 \item[(ii)] Then we define the normalized function 
\begin{equation}
\label{boston}
\psi_{n+1} = \frac{\psi_n^*}{\Norm{\psi_n^{*}}{L^2} }. 
\end{equation}
\end{itemize}
Note that the algorithm presented above preserves the real nature of the function $\psi$: if $\psi_0$ is real valued, then for all $n \in \N$, $\psi_n$ is real, and that the same holds true for symmetric functions satisfying $\psi_n(-x) = \psi_n(x)$.

Our results can be summarized as follows: 
\begin{itemize}
\item In the semi-discrete case described above, if the initial data $\psi_0$ is real, symmetric and sufficiently close to $\eta$, then $\psi_n$ converges exponentially towards $\eta$ in $H^1$. 
\item In the fully discrete case, where the discretization is space is made by finite difference in space with mesh $h$ combined with a Dirichlet cut-off for large values of $x$ (namely $x \leq Kh$), then we prove the exponential convergence towards a modified soliton $\eta_{h,K}$ that is $\mathcal{O}( h + \frac{1}{h^2} e^{- C_1 Kh})$ close to the exact soliton $\eta$. 
\end{itemize}

The main property explaining the excellent performance of this method is the fact that the linearly implicit numerical scheme exactly preserves the ground state: if $\psi_n = \eta$, then $\psi_{n+1} = \eta$, and the same holds true for discrete in space ground states satisfying a discrete version of \eqref{etacube}.  This fact is very general: it holds for any linearly implicit scheme applied to a semi-linear PDEs. This important feature makes of course this scheme much more attractive than other possible schemes where the nonlinearity would be approached by 
\begin{equation}
\label{otherd}
 -  \widehat {\nabla H}(\psi_n,\psi_n^*) =  \frac12\Delta\psi_n^* + 
 \left\{\begin{array}{ll}|\psi_n|^2\psi_n & \mbox{\rm(semi-explicit)},\\
  |\psi_n^*|^2\psi_n^{*} & \mbox{\rm (fully-implicit)}, 
 \end{array}
 \right.
\end{equation}
despite the fact that fully-implicit schemes enjoy the energy diminishing property for standard gradient system (see \cite{hairerlubich}). 
As we will discuss later, these schemes still converge, but towards modified ground states $\eta_\tau = \eta + \mathcal{O}(\tau)$. As these results would be weaker for longer detailed proofs, we only give the main arguments to obtain them, keeping a full detailed convergence proof for the linearly implicit scheme. 

\section{The Hamiltonian near the ground state}

We work in the space $V$ of real symmetric functions of $H^1$:
\ben
V:=\left\{\phi\in H^1(\R,\R)\,  | \, \phi(-x)=\phi(x)\right\}.
\een
We consider the usual $L^2$ and $H^1$ norms, and denote by $\scal{\cdot}{\cdot}$ the canonical real scalar product on $L^2$ :
\ben
\scal{\phi}{\psi}=\int_{\R}\phi(x)\psi(x)\dd s,\quad\Norm{\phi}{L^2}^2=\scal{\phi}{\phi},\quad
\Norm{\phi}{H^1}^2=\Norm{\phi}{L^2}^2+\Norm{\partial_x\phi}{L^2}^2.
\een

%We first introduce local coordinates around $\eta$ in $H^1$, necessary  to analyze the behavior of the numerical scheme. 
%Then we project this coordinate system onto the $L^2$ unit sphere and state the main coercivity result yielding to the classical orbital stability results. 

\subsection{Coordinates in the neighborhood of the ground state}
We introduce the set of the functions $R$-close to $\eta$
\ben
\mathcal{U}(R):=\left\{\phi\in V\, |\, \Norm{\phi-\eta}{H^1}<R\right\},
\een
the set
\ben
W:=\left\{u\in V \, | \, \scal{u}{\eta}=0\right\}
\een
and the map $\chi$
\ben
\begin{array}{rcl}
\chi : \R\times W & \rightarrow & V\\
(r,u) & \mapsto & (1+r)\eta+u.
\end{array}
\een
The map $\chi$ allows to use $(r,u)$ as coordinates in $\mathcal{U}(R)$. Observe that $\chi$ is smooth with bounded derivatives and so is the inverse $\chi^{-1}$, from the explicit formula
\ben
\begin{array}{rcl}
\chi^{-1} : V & \rightarrow & \R\times W\\
\psi & \mapsto & (r(\psi),u(\psi))= \left(\scal{\psi}{\eta}-1,\psi-\scal{\psi}{\eta}\eta\right).
\end{array}
\een
We will also use the following notations: We define
the $L^2$  projectors 
$$
P_\eta u = \langle u,\eta\rangle \eta
\quad \mbox{and}\quad
P_W u = I - P_\eta, 
$$
and we define the function 
\begin{equation}
\Hct(r,u) = (H \circ \chi) (r,u). 
\end{equation}
We can verify the following relations: 
\begin{equation}
\label{schimmel}
\partial_r \Hct(r,u) = \scal{\nabla H(\chi(r,u))}{\eta} = \eta^{-1} (P_\eta \nabla H) (\chi(r,u))  \in \R, 
\end{equation}

and
\begin{equation}
\label{steinway}
\nabla_u \Hct(r,u)  = P_W \nabla H ( \chi(r,u)) \in H^{-1}. 
\end{equation}

Before collecting some expressions of the Hamiltonian and the gradient flow in coordinates $(r,u)$, let us introduce the following notations: 
\begin{definition}
Let $p \geq 1$. 
We say that $R(u) = \mathcal{O}(\Norm{u}{H^1}^p)$ if for all $B$,  there exists a constant $C$ such that for all $u\in H^1$, $\Norm{u}{H^1} \leq B$, we have 
$$
\Norm{R(u)}{H^1}  \leq C \Norm{u}{H^1}^p.
$$
We will also use the notation 
$R(u,v) = \mathcal{O}(\Norm{u,v}{H^1}^p)$ if for all $B$, there exists $C$ such that for all $u,v \in H^1$ satisfying $\Norm{u}{H^1} \leq B$ and $\Norm{v}{H^1} \leq B$, then we have 
$$
\Norm{R(u,v)}{H^1}  \leq C \Big( \Norm{u}{H^1}^p + \Norm{v}{H^1}^p\Big). 
$$
Finally, a function $u = \mathcal{O}(r)$ if $\Norm{u}{H_1} \leq C r$ for $r$ small enough and a constant $C$ independent of $u$. 
\end{definition}
With these notations and the fact that $\psi = (1 + r) \eta + u$, we compute using \eqref{etacube} that 
\begin{eqnarray*}
- \nabla H (\psi) &=& \frac12 \Delta \psi + \psi^3\\
&=&  (1 + r) (- \eta^3 +  \lambda \eta) + \frac12\Delta u + u^3 \\
&&+ ( 1 + r)^3 \eta^3 + 3 (1 + r)^2 \eta^2 u
+ 3 (1 + r) \eta u^2\\
&=&  \lambda \eta + \frac12\Delta u   + 3  \eta^2 u
+ \mathcal{O}(r + \Norm{u}{H^1}^2). 
\end{eqnarray*}
Note that to obtain the bound, we have used the fact that $\eta \in H^1$, as well as the estimate  
\begin{equation}
\label{gadin}
\Norm{uv}{H^1} \leq C \Norm{u}{H^1}\Norm{v}{H^1}, 
\end{equation}
for two functions $u$ and $v$. 
We deduce using \eqref{schimmel} and \eqref{steinway} that 
\begin{eqnarray*}
\partial_r \Hct(r,u) &=& - \lambda - \frac12\scal{u}{\Delta \eta}   - 3  \scal{\eta^3}{ u}
+ \mathcal{O}(r + \Norm{u}{H^1}^2)  \\
&=& - \lambda - 2  \scal{\eta^3}{ u}
+ \mathcal{O}(r + \Norm{u}{H^1}^2), 
\end{eqnarray*}
as $\scal{u}{\eta} = 0$ and $\scal{\eta}{\eta} = 1$, and 
\begin{eqnarray*}
\nabla_u \Hct(r,u) &=& - P_W( \lambda \eta + \frac12\Delta u   + 3  \eta^2 u) 
+ \mathcal{O}(r + \Norm{u}{H^1}^2) \\
&=& -  \frac12\Delta u   - 3  \eta^2 u + 2  \scal{\eta^3}{ u} \eta + \mathcal{O}(r + \Norm{u}{H^1}^2). 
\end{eqnarray*}

% if $R$ is small enough: precisely, the following result holds (stated as Lemma 4.2 in \cite{BFG})
%\begin{proposition}
%There exists some constants $r_0$ and $\rho$ such that the application $\chi$ is smooth and bounded with bounded derivatives from $[-r_0,r_0]\times\mathcal{B}(\rho)$ to $V$, and such that $\chi^{-1}$ is well-defined on $\mathcal{U}(\rho)$, and smooth with bounded derivatives. 
%
%Moreover, there exists a constant $C$ such that for all $\psi\in\mathcal{U}(R)$, we have
%\ben
%\Norm{u(\psi)}{H^1}\leq C\Norm{\psi-\eta}{H^1}.
%\een 
%\end{proposition}

\subsection{Projection onto the unit $L^2$ sphere}
Let us define now the function $u\mapsto r(u)$ from $W$ to $\R$ by the implicit relation
$$\Norm{\chi(r(u),u)}{L^2}^2=1.$$
By explicit computation, we get
\be\label{eq:ru}
r(u)=-1+\sqrt{1-\Norm{u}{L^2}^2},
\ee
from which we deduce that $r(u)$ is well defined and smooth in a neighborhood of $0$ in $H^1$, and that moreover $|r(u)|=\mathcal{O}{(\Norm{u}{L^2}^2)}$ when $u$ is small enough. Hence, $u\mapsto\chi(r(u),u)$ is a local parametrization of $\mathcal{S} \cap V$ in a neighborhood of $\eta$, where
$$\mathcal{S}:=\left\{\psi\in V \, | \, \Norm{\psi}{L^2}=1\right\}.$$
Note that in this parametrization, $u = 0$ corresponds to the ground state $\eta$. 
%{\it ** en fait ici, je pense qu'on peut quasiment copier-coller la fin de la partie 4 de \cite{BFG} ?}
We then define 
\be \label{Hcal}
\Hcal(u):=\Hct(r(u),u) = H(\chi(r(u),u)).
\ee
The main result in \cite{Weinstein85}, see also \cite{Grill87,Grill90,FGJS}, is the following: 

\begin{proposition}\label{prop:d2Hcal}
The point 
$u=0$ is a non degenerate minimum of $\Hcal$: there exist some positive constants $c_0$ and $\rho_0$ such that
\be \label{eq:d2Hcal}
\forall v\in W\quad d^2 \Hcal (0).(v,v)\geq c_0\Norm{v}{H^1}^2.
\ee
\end{proposition}
Note that we have 
$$
\nabla_u r (u) = - \frac{u}{\sqrt{1 - \Norm{u}{L^2}^2}} = - u + \mathcal{O}( \Norm{u}{H^1}^3). 
$$
Hence, as $ r(u) = \mathcal{O}(\Norm{u}{H^1})$, we have 
\begin{eqnarray*}
\nabla_u \Hc(u) &=& \partial_r \Hct (r(u), u) \nabla_u r(u) + (\nabla_u \Hct) (r(u),u)\\
&=&  \lambda u  -  \frac12\Delta u   - 3  \eta^2 u + 2  \scal{\eta^3}{ u} \eta 
+ \mathcal{O}(\Norm{u}{H^1}^2) \\
&=&  P_W(  \lambda u  - \frac12\Delta u   - 3  \eta^2 u) 
+ \mathcal{O}(\Norm{u}{H^1}^2).  
\end{eqnarray*}
From the previous proposition, we deduce the following: 
\begin{corollary}
The operator $A: W \to W$ defined by 
\begin{equation}
\label{adele}
A u := P_W(  \lambda u  - \frac12\Delta u   - 3  \eta^2 u) 
\end{equation}
is $L^2$ symmetric and positive definite in $H^1$: We have 
\begin{equation}
\label{equivH1}
c\Norm{u}{H^1}^2 \leq \scal{u}{Au}\leq C \Norm{u}{H^1}^2
\end{equation}
for some constants $c$ and $C$, and 
$$
\nabla_u \Hc(u)  = A u  + \mathcal{O}(\Norm{u}{H^1}^2) . 
$$
\end{corollary}
Let us remark that the coercitivity relation \eqref{equivH1} combined with Cauchy-Schwartz inequality implies that 
$$
c\Norm{u}{H^1}^2 \leq \scal{u}{Au} \leq C \Norm{u}{L^2} \Norm{Au}{L^2} \leq C \Norm{u}{H^1} \Norm{Au}{L^2}
$$
for some constant $C$.  
We thus infer the existence of a constant $c > 0$ such that 
\begin{equation}
\label{monique}
\Norm{Au}{L^2} \geq c \Norm{u}{H^1} \quad \mbox{and} \quad \Norm{Au}{L^2}^2 \geq c \scal{u}{Au}. 
\end{equation}
%\begcin{proof}
%The result \eqref{eq:d2Hcal} can be shown for $u=0$ by the general method of \cite{FGJS} : they prove the same result but in a more general case in their appendix B,C and D. The coercivity estimate extends then uniformly to a neighborhood of $0$ given that $\Hcal$ is smooth with bounded derivatives.
%\end{proof}
%
%\textcolor{blue}{
%FAUX
%\subsection{An estimate of $\nabla H$ in the neighborhood of the ground state}
% We use many times in the following this short result :
%\begin{lemma} There exists $C$ such that for all $\psi=\chi(r,u)$ in $\mathcal{U}(\rho_0)$, 
%\be\label{eq:dvtH}
%\Norm{\nabla H(\chi(r,u))+\lambda \eta}{L^2}\leq 
%\left\{\begin{array}{l}
%C(|r|+\Norm{u}{L^2}) ,\\
%%\Norm{\nabla H(\psi)+\lambda e^{i\alpha}\eta}{L^2}\leq 
%C\Norm{u}{L^2}\qquad\text{if }\Norm{\psi}{L^2}=1.
%\end{array}\right.
%\ee
%\end{lemma}
%\begin{proof}
%For any $\phi$ in $H^1$,
%\begin{eqnarray*}
%\scal{\nabla H(\psi)}{\phi}&=&\scal{\nabla H(\eta)}{\phi}+\int_0^1d^2H(\eta+s(\psi-\eta)).(\psi-\eta,\phi)ds\\
% &=&-\lambda\scal{\eta}{\phi}+ \int_0^1d^2H\left(\eta+s(r\eta+u)\right).\left(r\eta+u,\phi\right)ds
%\end{eqnarray*}
%Given that $H$ is multilinear and continuous in $H^1$, we get that $d^2H$ is smooth on $\mathcal{U}(\rho_0)$ and thus the existence of $C$.
%\end{proof}}

%%%%%%%%%%%%%%%%%%%%%%%%%%%%%%%%%%%%%%%%%%%%%%%%%%%%%%%%%%%%%%%%%%%%%%%%%%%%%%%%%%%%%%%%%%%%%%%%%%%%%%%%%%%%%%%%%%%%%%%%%%%%%%%%%%%%%%%%%%%%%%%%%%%%%%%%%%%%%%%%%%%%
\section{Continuous normalized gradient flow}

We consider the continuous normalized gradient flow (see for instance \cite{bao})
\begin{equation}\label{CNGF}
\partial_t\psi=-\nabla H(\psi)+\scal{\nabla H(\psi)}{\frac{\psi}{\Norm{\psi}{L^2}}}\frac{\psi}{\|\psi\|_{L^2}}, 
\end{equation}
which is the projection of the standard gradient flow $\partial_t\psi=-\nabla H(\psi)$ onto the unit $L^2$ sphere. 
The local existence of an $H^1$ solution to \eqref{CNGF} is guaranteed by standard argument, using the fact that $\Delta$ defines a semi group in $H^1$, and that $H^1$ is an algebra in dimension $1$ (see \eqref{gadin}).

%%%%%%%%%%%%%%%%%%%%%%%%%%%%%%%%%%%%%%%%%%%%%%%%%%%%
\subsection{The gradient flow in local variables}

We assume for the moment that $\psi(t)$ remains in a ball sufficiently close to $\eta$ so that we can write $\psi(t) = ( 1 + r(t)) \eta + u(t)$ with $r(t) > -1$ and $u(t) \in W$. We have in this case $P_\eta \psi = ( 1 + r) \eta$ and $P_W \psi = u$. 
Note also that 
\begin{eqnarray*}
\scal{\nabla H(\psi)}{\psi}
&=& 
\langle P_\eta \nabla H, P_\eta \psi \rangle + \langle P_W \nabla H, P_W \psi \rangle\\
&=&  ( 1 + r) \partial_r \Hct  + \scal{\nabla_u \Hct}{u}
\end{eqnarray*}
and that
$$
\Norm{\psi}{L^2}^2 = ( 1+ r)^2 + \Norm{u}{L^2}^2. 
$$
It turns out that the relation  $\Norm{\psi}{L^2}^2 = 1$ implies that $\Norm{u}{L^2} \leq 1$. In the following we will only work with functions $u$ satisfying this condition. 
Applying successively the operators $P_\eta$ and $P_W$ to the equation \eqref{CNGF}, we obtain the relation 
\begin{eqnarray*}
\partial_t r &=& - \partial_r \Hct (r,u) + \frac{( 1 + r)^2}{( 1+ r)^2 + \Norm{u}{L^2}^2} \partial_r \Hct  + \frac{( 1 + r)}{( 1+ r)^2 + \Norm{u}{L^2}^2} \scal{\nabla_u \Hct}{u},\\
\partial_t u &=& - \nabla_u \Hct (r,u) + \frac{( 1 + r) u }{( 1+ r)^2 + \Norm{u}{L^2}^2} \partial_r \Hct  + \frac{u }{( 1+ r)^2 + \Norm{u}{L^2}^2} \scal{\nabla_u \Hct}{u}.
\end{eqnarray*}

As the $L^2$ norm of $\psi$ is preserved, we calculate directely that $(1 + r)^2 + \Norm{u}{L^2}^2$ is preserved along the flow of this system (and is constant equals to one), and that we can solve $r$ in terms of $u$ as above (see \eqref{eq:ru}). 
We thus obtain the closed equation 
\begin{equation}
\label{weber}
\partial_t u =- \nabla_u \Hct (r(u) ,u) + ( 1 + r(u) ) u \, \partial_r \Hct(r(u),u)  +u \scal{\nabla_u \Hct}{u}. 
\end{equation}
But we have 
$$
(\nabla_u \Hct) (r(u),u) = \nabla_u \Hc(u) + \partial_r \Hct (r(u), u) \frac{u }{\sqrt{1 - \Norm{u}{L^2}^2}}, 
$$
hence we can write the equation \eqref{weber} as 
\begin{eqnarray*}
\partial_t u &=& - \nabla_u \Hc(u) + u \langle u, \nabla_u \Hc(u) \rangle   - \partial_r \Hct (r(u), u) \frac{u }{\sqrt{1 - \Norm{u}{L^2}^2}}  \\
&&
+ u \sqrt{1 - \Norm{u}{L^2}^2} \partial_r \Hct(r(u),u)
 + \partial_r \Hct (r(u), u) \frac{ u \Norm{u}{L^2}^2 }{\sqrt{1 - \Norm{u}{L^2}^2}}. 
\end{eqnarray*}
and hence
\begin{equation}
\label{tard}
\partial_t u = - \nabla_u \Hc(u) 
+ u \langle u, \nabla_u \Hc(u) \rangle,  
\end{equation}
which is the gradient flow in local coordinates $(r(u),u)$ on $\mathcal{S}$.  
Note that with the notation of the previous paragraph, we we can write this equation as follows: 
$$
\partial_t u = - A u + R(u)
$$
with $R(u) = \mathcal{O}(\Norm{u}{H^1}^2)$. 
\subsection{Convergence of the flow}
We prove now that the solution of the normalized gradient flow \eqref{CNGF} converges towards the ground state $\eta$ if the initial value sufficiently close to it. 
\begin{theorem}
There exists $\mu > 0$ such that 
if $u(t) \in W$ is a solution of \eqref{tard} with $\Norm{u_0}{H^1}$ sufficiently small, then we have 
$$
\Norm{u(t)}{H^1} \leq C(u_0) e^{- \mu t }
$$
for all $t >0$, and 
for some constant $C(u_0)$ depending on $u_0$. 

Hence, if $\psi(t) \in V \cap \Sc$ is a solution of \eqref{CNGF} such that $\Norm{\psi(0)-\eta}{H^1}$ is small enough, then for all $t$
\ben
\Norm{\psi(t)-\eta}{L^2} \leq C e^{- \mu t}
\een
for some constant $C$ and $\mu$. 
\end{theorem}

\begin{proof}
As $A$ is symmetric, we calculate that 
$$
\partial_t \scal{u(t) }{Au(t)} = - 2\scal{Au(t) }{Au(t)} + 2 \scal{R(u(t))}{A u(t)}. 
$$
Using \eqref{monique}, 
and the fact that for all $u$ and $v$ in $H^1$, $\scal{u}{Av} \leq C \Norm{u}{H^1} \Norm{v}{H^1}$, we obtain 
$$
|\partial_t \scal{u(t) }{Au(t)}| \leq  - c \scal{u(t) }{Au(t)} + 2 C \Norm{R(u(t))}{H^1} \Norm{u(t)}{H^1}  
$$
for some positive constant $c$ and $C$. 
As we have by definition $\Norm{R(u(t))}{H^1} \leq \Norm{u(t)}{H^1}^2 \leq C \scal{u(t)}{A u(t)}$, 
we infer the estimate 
$$
|\partial_t \scal{u(t) }{Au(t)} |Ê\leq  - c\scal{u(t)}{A u(t)} +C\scal{u(t)}{A u(t)}^{3/2}. 
$$
If $\scal{u(t) }{Au(t)} \leq B^2$ at $ t= 0$, we check that by continuity argument, we have the rough estimate
$$
\scal{u(t) }{Au(t)} \leq B^2 e^{ - (c - C B) t},  
$$
The result easily follows by taking $B$ small enough and using \eqref{equivH1}. 
\end{proof}

%%%%%%%%%%%%%%%%%%%%%%%%%%%%%%%%%%%%%%%%%%%%%%%%%%%%%%%%%%%%%%%%%%%%%%%%%%%%%%%%%
%%%%%%%%%%%%%%%%%%%%%%%%%%%%%%%%%%%%%%%%%%%%%%%%%%%%%%%%%%%%%%%%%%%%%%%%%%%%%%%%%%
%%%%%%%%%%%%%%%%%%%%%%%%%%%%%%%%%%%%%%%%%%%%%%%%%%%%%%%%%%%%%%%%%%%%%%%%%%%%%%%%%%
\section{Convergence result}

We consider the scheme described in the introduction: from a function $\psi_n$ of $L^2$ norm $1$, close enough to $\eta$, we can write $\psi_n = (1 + r(u_n)) \eta + u_n$ with $u_n \in W$. We then define $\psi_{n+1}$ by the relations \eqref{poulguen} and \eqref{boston}. 
%\begin{eqnarray}
%\label{wald}
%\psi_n^*  &=& \psi_n + \tau \left(  \frac12\Delta\psi_n^* + |\psi_n|^2\psi_n^{*}\right)\nonumber\\
%&=& \psi_n - \tau   \widehat{\nabla H}(\psi_n,\psi_n^*)
%\end{eqnarray}
%and then we set $\psi_{n+1} = \psi_n\slash \Norm{\psi_n}{L^2}$. 
Note that as $\psi_{n+1}$ is of $L^2$ norm $1$, we can define $u_{n+1}$ by the formula 
$ \psi_{n+1} = (1 + r(u_{n+1})\eta + u_{n+1}$ if $\psi_{n+1}$ is close enough to $\Gamma$. The map $u_n \mapsto u_{n+1}$ satisfies the following relation:

\begin{lemma}
Let $R \leq \frac{1}{4}$ be given. There exists $\tau_0$ such that for all $\tau \leq \tau_0$
the numerical scheme \eqref{poulguen}-\eqref{boston} is well defined from $\mathcal{S}\cap \mathcal{U}(R)$ to $\mathcal{S}\cap \mathcal{U}(2R)$ and is equivalent to an application $u_n \mapsto u_{n+1}$ satisfying 
\begin{equation}
\label{wiener}
u_{n+1}  = u_n  +  \tau P_W \big(  \frac{1}{2}\Delta u_{n+1}
- \lambda u_n +   \eta^2   u_{n+1} +  2\eta^2 u_n     \big)  + \mathcal{O}(\tau \Norm{u_n, u_{n+1}}{H^1}^2), 
\end{equation}
which is a time discretization of the equation 
\begin{eqnarray*}
\partial_t u &=& P_W(  - \lambda u + \frac12\Delta u   + 3  \eta^2 u ) +  \mathcal{O}(\Norm{u}{H^1}^2)\\
&=&  - A u  +  \mathcal{O}(\Norm{u}{H^1}^2),
\end{eqnarray*}
corresponding to the normalized gradient flow system \eqref{CNGF}. 
\end{lemma}

\begin{proof}
%Note that as $\Norm{\psi_n}{L^1}^2 = (1 + r(u_n))^2 + \Norm{u_n}{L^2}^2 = 1$, we have $\Norm{u_n}{L^2} \leq 1$. 
Note that with the assumption on $R$, we can assume that $\Norm{u_n}{H^1} \leq 1/2$. 
Let us calculate $(r_n^*,u_n^*)$ such that 
$$
\psi_n^* = \chi(r_n^*,u_n^*) = ( 1+ r_n^*) \eta + u_n^*, 
$$
with $u_n^* \in W$, that is $\scal{\eta}{u_n^*} = 0$. As $r_n = r(u_n) = \mathcal{O}(\Norm{u_n}{H^1}^2)$, we have $|\psi_n|^2 =  \eta^2 + 2u_n \eta + \mathcal{O}(\Norm{u_n}{H^1}^2)$. 
Let us assume that $|r_n^*| \leq 1$, we can expand the Hamiltonian term 
\begin{multline*}
 -\widehat{\nabla H}(\psi_n,\psi_n^*) = \frac12\Delta\psi_n^* + |\psi_n|^2\psi_n^{*}\\
 %&=& ( 1+ r_n^*) ( -\eta^3 + \lambda \eta)  + \frac{1}{2}\Delta u_n^*
 %+ | \eta + u_n|^2( ( 1+ r_n^*) \eta + u_n^*) + \mathcal{O}(\Norm{u_n}{H^1}^2)\\
 = ( 1+ r_n^*) ( -\eta^3 + \lambda \eta)  + \frac{1}{2}\Delta u_n^* +  \eta^2 u_n^* + (\eta^3 + 2 u_n \eta^2) ( 1+ r_n^*)   + \mathcal{O}(\Norm{u_n,u_n^*}{H^1}^2), 
\end{multline*}
and hence we find
\begin{equation}
\label{li}
-\widehat{\nabla H}(\psi_n,\psi_n^*) = \lambda \eta ( 1+ r_n^*)   + \frac{1}{2}\Delta u_n^*
 + \eta^2   u_n^*
+  2 u_n \eta^2 ( 1+ r_n^*)  + \mathcal{O}(\Norm{u_n, u_n^*}{H^1}^2). 
\end{equation}
By taking the projection $P_\eta$ of the equation \eqref{poulguen}, we obtain 
$$
1 + r_n^* = 1+  r_n  + \tau ( 1+ r_n^*)\lambda 
+ 2 \tau \scal{u_n}{ \eta^3} ( 1+ r_n^*)  + \mathcal{O}(\tau \Norm{u_n, u_n^*}{H^1}^2). 
$$
\begin{remark}
\label{remi}
At this stage, the same calculations for the implicit explicit of fully implicit schemes would yield a term depending on $ \eta^3$ in \eqref{li} which is not in $W$ nor in $\R \eta$. After taking a projection, this term would not vanish, while the ``constant" term  $\lambda \eta ( 1+ r_n^*)$ is here orthogonal to $W$ (which reflects the fact that the ground state is exactly integrated by the linearly implicit scheme). 
\end{remark}
We deduce that for $\tau \leq \tau_0$ small enough with respect to $R$ (and $\lambda$) 
\begin{eqnarray*}
1 + r_n^* &=& \frac{1 + r_n}{1 - \tau \lambda - 2 \tau \scal{u_n}{ \eta^3}}+ \mathcal{O}(\tau \Norm{u_n, u_n^*}{H^1}^2) \\
&=&   \frac1{1 - \tau \lambda}{}\big( 1 + r_n + \frac{2 \tau}{1 - \tau \lambda} \scal{u_n}{ \eta^3} \big) +  \mathcal{O}(\tau \Norm{u_n, u_n^*}{H^1}^2), 
\end{eqnarray*}
and by taking the projection on $W$, we obtain 
\begin{equation}
\label{kreisleriana}
u_n^* = u_n + \tau P_W \big(  \frac{1}{2}\Delta u_n^*
 + \eta^2   u_n^*
+  \frac{2}{1 - \tau \lambda} u_n \eta^2  \Big)+ \mathcal{O}(\tau \Norm{u_n, u_n^*}{H^1}^2).  
\end{equation}
Let us now calculate and asymptotic expansion of the $L^2$ norm of $\psi_n^*$, 
$$
\Norm{\psi_n^*}{L^2}^2 =  (1 + r_n^*)^2 + \Norm{u_n^*}{L^2}^2. 
$$
We have 
\begin{eqnarray*}
( 1 + r_n^*)^2 &=&   \left( \frac1{1 - \tau \lambda}{}\right)^2\big( 1 + r_n + \frac{2 \tau}{1 - \tau \lambda} \scal{u_n}{ \eta^3} \big)^2+  \mathcal{O}(\tau \Norm{u_n, u_n^*}{H^1}^2) \\
&=&\left( \frac1{1 - \tau \lambda}{}\right)^2\big( (1 + r_n)^2+ \frac{4 \tau}{1 - \tau \lambda} \scal{u_n}{ \eta^3} \big)+  \mathcal{O}(\tau \Norm{u_n, u_n^*}{H^1}^2) 
\end{eqnarray*}
and
$$
\Norm{u_n^*}{L^2}^2 = \Norm{u_n}{L^2}^2  + \mathcal{O}(\tau \Norm{u_n, u_n^*}{H^1}^2).  
$$
Hence we calculate that 
\begin{equation*}
\begin{split}
\Norm{&\psi_n^*}{L^2}^{-1}=\Big( ( 1 + r_n^*)^2 + \Norm{u_n^*}{L^2}^2 \Big)^{-1/2}\\
%&&= (1 - \tau \lambda) \left( \big( (1 + r_n)^2+ \frac{4 \tau}{1 - \tau \lambda} \scal{u_n}{ \eta^3} \big) + (1 - \tau \lambda)^2 \Norm{u_n}{L^2}^2 +\mathcal{O}(\tau \Norm{u_n, u_n^*}{H^1}^2)\right)^{-1/2}\\
&= (1 - \tau \lambda) \left( (1 + r_n)^2+ \frac{4 \tau}{1 - \tau \lambda} \scal{u_n}{ \eta^3}  + \Norm{u_n}{L^2}^2 +\mathcal{O}(\tau \Norm{u_n, u_n^*}{H^1}^2)\right)^{-1/2}
\\
&= (1 - \tau \lambda) \left(1 + \frac{1}{1 - \tau \lambda}4 \tau \scal{u_n}{ \eta^3}   +\mathcal{O}(\tau \Norm{u_n, u_n^*}{H^1}^2)\right)^{-1/2}
\\
&=  1 -\tau \lambda - 2 \tau \scal{u_n}{ \eta^3}   +\mathcal{O}(\tau \Norm{u_n, u_n^*}{H^1}^2). 
\end{split}
\end{equation*}
Let us rewrite \eqref{kreisleriana} as 
$$
u_n^* - \tau P_W \big(  \frac{1}{2}\Delta u_n^*
 + \eta^2   u_n^* \big) = u_n + \tau P_W \big(  \frac{2}{1 - \tau \lambda} u_n \eta^2  \Big)+ \mathcal{O}(\tau \Norm{u_n, u_n^*}{H^1}^2) . 
$$
Hence by multiplying by $\Norm{\psi_n^*}{L^2}^{-1}$ we obtain 
\begin{multline*}
u_{n+1} - \tau P_W \big(  \frac{1}{2}\Delta u_{n+1}
 + \eta^2   u_{n+1} \big) \\
 = u_n  -\tau \lambda u_n + \tau P_W ( 2 u_n \eta^2)   + \mathcal{O}(\tau \Norm{u_n, u_{n+1}}{H^1}^2) , 
\end{multline*}
which yields the result. 
\end{proof}
We can now prove the main result of this section: 
\begin{theorem}
There exists constants $R$,  $\tau_0$, $r$ and $C$ such that if $\psi_0 \in \mathcal{S} \cap \mathcal{U}(R)$, that for all $\tau \leq \tau_0$ 
the solution of the numerical scheme \eqref{poulguen}-\eqref{boston} satisfies 
\begin{equation}
\label{violin}
\forall\,n \quad  \Norm{\psi_n - \eta}{H^1} \leq C e^{- r n \tau}. 
\end{equation}
\end{theorem}
\begin{proof}Note that with the previous notations, we have 
$\Norm{\psi_n - \eta}{H^1} = \mathcal{O}(\Norm{u_n}{H^1})$. 
We will prove the following result, which implies \eqref{violin} after a slight change of constants:  
There exists $B$, $\tau_0$, $r$ and $C$ such that if $\Norm{u_0}{H^1} \leq B$ and $\tau \leq \tau_0$, then every sequence $(u_n)_{n \geq 0}$ defined by \eqref{wiener} satisfies the estimate
$$
\forall\,n \quad  \Norm{u_n}{H^1} \leq C e^{- r n \tau}. 
$$
We rewrite the system \eqref{wiener} as 
$$
u_{n+1}  = u_n  -  \tau (L u_{n+1} + B u_n) + \mathcal{O}(\tau \Norm{u_n, u_{n+1}}{H^1}^2), 
$$
with 
$$
L u = P_W(\textstyle  \frac12\Delta u  + \eta^2 u) \quad \mbox{and}  \quad B u = P_W(  - \lambda u + 2   \eta^2 u ). 
$$
We know that $A = L + B$ is coercive in $H^1$. Note moreover that $B$ is a symmetric bounded operator in $W$.  
We can write 
$$
u_{n+1}  = u_n  -  \tau (L + B) u_{n+1} +\tau  B (u_n - u_{n+1}) + \mathcal{O}(\tau \Norm{u_n, u_{n+1}}{H^1}^2). 
$$
Now for $\tau$ sufficiently small, the operator $I + \tau B$ is symmetric, invertible, positive and bounded from $W$ to itself. Hence we can rewrite the numerical scheme as 
$$
u_{n+1}  = u_n  -  \tau ( I + \tau B)^{-1}(L + B) u_{n+1} + \mathcal{O}(\tau \Norm{u_n, u_{n+1}}{H^1}^2). 
$$
Setting 
$v_n = ( I + \tau B)^{1/2} u_n$, 
we obtain the relation
$$
v_{n+1}  = v_n  -  \tau ( I + \tau B)^{-1/2}(L + B)( I + \tau B)^{-1/2} v_{n+1} + \mathcal{O}(\tau \Norm{v_n, v_n^*}{H^1}^2). 
$$
In this new variable $v_n$, we thus can write the induction relation 
\begin{equation}
\label{leon}
v_{n+1}  = v_n  -  \tau A_\tau  v_{n+1} + \tau w_n 
\end{equation}
with $ A_\tau = ( I + \tau B)^{-1/2}(L + B)( I + \tau B)^{-1/2}$, 
\begin{equation}
\label{beaulieu}
w_n =  \mathcal{O}(\Norm{v_n, v_{n+1}}{H^1}^2) 
\quad\mbox{and}\quad 
 c \Norm{v}{H^1}^2 \leq \scal{A_\tau v}{v}_{L^2} \leq C \Norm{v}{H^1}^2, 
\end{equation}
for some constants $c$ and $C$ independent of $\tau$. 
Note also that the inequality \eqref{monique} holds for the operators $A_\tau$ uniformly in $\tau \leq \tau_0$ sufficiently small. 
From \eqref{leon}, we can write
\begin{multline*}
 \scal{A_\tau v_n + \tau A_\tau w_n }{v_n + \tau w_n}_{L^2}
=
\scal{A_\tau v_{n+1} + \tau A_\tau^2 v_{n+1}}{ v_{n+1} + \tau A_\tau v_{n+1} }_{L^2} \\
= \scal{A_\tau v_{n+1}}{v_{n+1}}  +2  \tau \scal{A_\tau  v_{n+1}}{A_\tau v_{n+1}}_{L^2}   + \tau^2 \scal{A_\tau^2 v_{n+1}}{A_\tau v_{n+1}}_{L^2}
\\
\geq ( 1 + 2 c \tau) \scal{A_\tau v_{n+1}}{v_{n+1}},
\end{multline*}
using \eqref{monique} and the fact that $\scal{A_\tau^2 v_{n+1}}{A_\tau v_{n+1}}_{L^2} \geq 0$. 

On the other hand we have 
$$
\scal{A_\tau v_n + \tau A_\tau w_n }{v_n + \tau w_n}_{L^2} =  \scal{A_\tau v_n}{v_n}  
+ 2 \tau \scal{A_\tau v_n}{w_n} + \tau^2 \scal{A_\tau w_n}{w_n}.
$$
But using an integration by part in the unbounded part of $A_\tau$, we have for all $\tau \leq \tau_0$ sufficiently small, and for any $v,w$ in $H^1$
$$
|\scal{A_\tau v}{w}| \leq  C \Norm{v}{H^1}\Norm{w}{H^1}. 
$$
Using the estimate on $w_n$, we thus see that if $\scal{A_\tau v_n}{v_n} $ and $\scal{A_\tau v_{n+1}}{v_{n+1}}$ are bounded by $B^2$, we have 
$$
\scal{A_\tau v_n + \tau A_\tau w_n }{v_n + \tau w_n}_{L^2} \leq  (1  + C \tau B) \scal{A_\tau v_n}{v_n}  + C \tau B \scal {A_\tau v_{n+1}}{v_{n+1}}.
$$
The previous inequalities imply: 
$$
\scal{A_\tau v_{n+1}}{v_{n+1}}_{L^2} \leq \frac{1 + C B \tau}{1 + c \tau - C B \tau}\scal{A_\tau v_{n}}{v_{n}}_{L^2}
$$
as long as $\scal{A_\tau v_n}{v_n} \leq B^2$ and $\scal{A_\tau v_{n+1}}{v_{n+1}} \leq B^2$. 
If $B$ is small enough (independently of $\tau$) we have 
$$
 \frac{1 + C B \tau}{1 + c \tau - C B \tau} \leq e^{- a \tau}
 $$
 for some constant $a$ independent of $\tau$. 
 This implies that 
 $$
\scal{A_\tau v_{n}}{v_{n}}_{L^2} \leq B^2 e^{- a n \tau},
 $$
for all $n$, by an induction argument. This shows the result. 
 \end{proof}

\begin{remark}
\label{rutgers}
If we choose another discretization - implicit explicit or fully implicit, see \eqref{otherd} - then as explained in Remark \ref{remi} the term in $\eta^3$ do not disapear in the equation, and after some manipulations we end up with a recursion formula of the form 
$$
v_{n+1} = v_n + \tau^2 g - \tau A_\tau v_{n+1} + \tau \rho_n
$$
with $\rho_n = \mathcal{O}( \Norm{v_n,v_{n+1}}{H^1}^2 ) $ and $g \in W$ a non zero function. It is relatively easy to prove that in this case, $v_n$ converges towards a function $v_\tau^\infty = \tau A_{\tau}^{-1} g + \mathcal{O}(\tau^2)$. Indeed, the previous equation if a discretization of the modified problem
$$
\partial_t v = \tau g - A_\tau v + B(v) 
$$
whose solution exponentially converges towards the solution of the problem
$$
 \tau g - A_\tau v + B(v)  = 0  
$$
which exists by standard implicit function theorem in $H^1$. This shows that in this case, $\psi_n$ converges exponentially in $H^1$  towards a modified soliton $\eta_{\tau} = \eta + \mathcal{O}(\tau)$. As the result is weaker than for the linearly implicit scheme, we do not give more mathematical details. 
\end{remark}

\section{Fully discrete case}

The goal of this last section is to show that the previous results carry over to fully discrete systems. This fact relies essentially on the results in \cite{BP10} and \cite{BFG}: After discretization of the previous system by finite difference method, there exists a discrete ground state $\eta_h$ minimizing a discrete convex functional $\Hc_h$ which is an approximation of the exact functional $\Hc$ defined above. Here  $h$ denotes the space discretization parameter and $\eta_h$ is close to $\eta$ with an error of order $h$. Moreover,  the same holds true after Dirichlet cut-off upon adding an exponentially decreasing error term, and it can be proven that the discrete functionals satisfy the same estimates as the continuous ones, uniformly with respect to the discretization parameters.

Having fixed a positive parameter $h$ we discretize in space by
substituting the sequence
$\psi^{\l}\simeq\psi(h \l)$, $\l\in\Z$ for the function $\psi(x)$,  and the second order operator of finite difference
$\Delta_h$ defined by
\begin{equation}
\label{deltamu}
(\Delta_h \psi)^{\l}:=\frac{\psi^{\l+1}+\psi^{\l-1}-2\psi^{\l}}{h^2}, 
\end{equation}
for the Laplace operator   $\partial_{xx}$. We also impose Dirichlet boudary conditions for $|j| \geq K+1$, and 
the parabolic equation \eqref{telephone} then becomes 
\begin{equation}
\left\{
\begin{array}{rcl}
\label{dnls} \displaystyle \frac{\dd}{\dd t}
\psi^\l &=& \displaystyle\frac{1}{2 h^2}(\psi^{\l+1}+\psi^{\l-1}-2\psi^\l)+ 
|\psi^\l|^2\psi^\l, \quad \ell \in\Z, \quad |\ell| \leq K,\\[2ex]
\psi^\ell &=& 0,\quad |\ell| \geq K+1,
\end{array}
\right.
\end{equation}
where $t \mapsto \psi(t) = (\psi^\ell(t))_{\ell \in \Z}$ is an application from $\R$ to $\R^\Z$.  With this equation is associated a Hamiltonian function and a discrete $L^2$ norm given by 
\begin{equation}
\label{eq:discrham}
H_h(\psi)=h\sum_{j\in \Z}\left[\left|\frac{\psi^j
-\psi^{j-1}}{h}\right|^2
-\frac{|\psi^j|^4} {2}\right]
\quad \mbox{and}\quad
N_h(\psi)=h\sum_{j\in \Z} |\psi^j|^2. 
\end{equation}
The discrete space of functions is 
$$
V_{h,K} = \{ \psi^j \in \C^{\Z}\, | \, \psi^j = \psi^{-j}, \quad \psi^j = 0 \quad \mbox{for}\quad j > K\}
$$
equipped with the discrete $H^1$ norm 
\begin{equation}
\label{h1}
\Norm{\psi}{h}^2 = 2h \sum_{j\in\Z} \frac{|\psi^{j+1}-\psi^j |^2}{h^2}+h \sum_{j\in\Z} |\psi^j|^2, 
\end{equation}
and we also define the scalar product 
$$
\langle \psi, \varphi\rangle_h := h \sum_{j \in \Z} \psi^j \varphi^j. 
$$
Following \cite{BFG,BP10}, we identify $V_h$ with a finite element subspace of
$H^1(\R)$. More precisely,  defining the function $s:\R \to \R$ by 
\begin{equation}
\label{fe.1}
s(x)=
\begin{cases}
0\qquad\qquad  &{\rm if}\quad |x|>1,\\
x+1\quad  &{\rm if}\quad   -1\leq x\leq 0,\\
- x+1\quad  &{\rm if}\quad  0\leq x\leq 1,
\end{cases}
\end{equation}
 the identification is done through the map $i_h: V_h \to H^1(\R)$ defined by 
\begin{equation}
\label{eq:ih}
\left\{\psi_j\right\}_{j\in\Z}\mapsto (i_h\psi)(x) := \sum_{j\in\Z} \psi_j s \left(\frac{x}{h}-j\right) .
\end{equation}
The main results in \cite{BFG} and \cite{BP10} can expressed as follows: 
\begin{theorem}
For $h$ sufficiently small and $K$ sufficiently large, 
there exists a unique real minimizer $\eta_{h,K} = (\eta_{h,k}^j)_{j \in \Z} \in V_{h,K}$ of the functional $H_h(\psi)$ over the set
$$
\{Ê\psi = (\psi^j)_{j \in \Z} \in V_{h,K}, \quad N_h(\psi) = 1\,\}, 
$$
equipped with the norm \eqref{h1}. 
Moreover, $\eta_{h,K}$ satisfies the equation 
\begin{eqnarray}
\label{zglob}
\textstyle\frac12 (\Delta_h \eta_{h,K})^\ell + |\eta_{h,K}^\ell|^2 \eta_{h,K}^\ell &=& \lambda_h \eta_{h,K}^\ell, \quad \ell = -K,\ldots,K,\\[2ex]
\eta_{h,K}^\ell &=& 0, \quad |\ell| \geq K+1, \nonumber
\end{eqnarray}
for some $\lambda_h > 0$. 
Moreover, we can define a local coordinate system $\psi = \chi_h(r,u) = (1 + r) \eta_{h,K} + u$ in a vicinity of $\eta_{h,K}$ in $V_{h,K}$ with $u \in W_h := \{Ê\, u \in V_{h,K},\quad \scal{u}{\eta_{h,K}}_h = 0\}$, and such that if $r_h(u)$ is defined by the implicit relation $N_h( \chi_h(r_h(u),u))) = 1$, then the functional $\Hc_{h}(u) = H_h(\chi_h(r_h(u),u)))$ has a unique non degenerate minimum in $u = 0$ in $V_{h,K}$ and satisfies a convexity estimate of the form \eqref{eq:d2Hcal} uniformly in $h$ and $K$. 

Finally, the discrete ground state $\eta_{h,K}$ is close to the continuous one in the sense that 
\begin{equation}
\Norm{i_h \eta_{h,K} - \eta}{H^1} \leq C \Big( h + \frac{1}{h^2} e^{- C_1 hK } \Big), 
\end{equation}
where $C$ and $C_1$ do not depend on $h$, and where $i_h$ is the embedding \eqref{eq:ih}. 
\end{theorem}

The fully discrete algorithm corresponding to \eqref{poulguen}-\eqref{boston} then consists in the two following steps: From $(\psi^\ell_n)_{j = -K}^K$ such that $N_h(\psi_n) = 1$, 

\begin{itemize}
\item[(i)] Compute $\psi^*_n = (\psi^{*,\ell}_n)_{j \in \Z}$ defined by the relation 
\begin{equation}
\label{poulguend}
\left\{
\begin{array}{rcl}\displaystyle \psi^{*,\ell}_n &=& \psi^{\ell}_n + \tau \left(\frac12 (\Delta_h \psi_{n}^*)^\ell + 
|\psi^\l_n|^2\psi^{*,\l}_n\right), \quad \ell \in\Z, \quad |\ell| \leq K,\\[2ex]
\psi^{*,\ell}Ñ &=& 0,\quad |\ell| \geq K+1.
\end{array}
\right.
\end{equation}
\item[(ii)] Normalization step: 
\begin{equation}
\label{bostond}
\psi_{n+1}^\ell = \frac{\psi_n^{*,\ell}}{N_h(\psi_n^*)}.
\end{equation}
\end{itemize}

The equation \eqref{zglob} guarantees that this fully discrete algorithm preserves exactly the discrete ground state $\eta_{h,K}$. 
We can now copy the proof of the continuous case an adapt it directly to the fully discrete case: We can prove that 
$$
\Norm{\psi_{n} - \eta_{h,K}}{h} \leq C e^{- cn \tau}, \quad n \geq 0, 
$$
for some constants $C$ and $c$ independent of $h$, $K$ and $\tau$. By using the previous estimate, we obtain the following result: 

\begin{theorem}
There exist constants $B$, $C$, $C_1$ and $c$ such that for $h$, $\tau$ sufficiently small and $K$ sufficiently large, if $(\psi_0) = (\psi^\ell_0)_{\ell = -K}^{K}$ satisfies 
$$
\Norm{i_h \psi_0 - \eta}{H^1} \leq B
$$
then the solution $\psi_n$ of the fully discrete algorithm \eqref{poulguend}-\eqref{bostond} satisfies
$$
\Norm{i_h \psi_n - \eta}{H^1} \leq C \Big(e^{-c n \tau} + h +  \frac{1}{h^2} e^{- C_1 hK } \Big). 
$$
\end{theorem}

\begin{remark}
Following Remark \ref{rutgers}, we can prove a similar result for the implicit-explicit and fully implicit algorithms \eqref{otherd}, but the fully discrete schemes will converge towards a discrete ground state $\eta_{\tau,h,K}$ satisfying 
$$
\Norm{i_h \eta_{\tau,h,K} - \eta}{H^1} \leq C \Big(\tau + h +  \frac{1}{h^2} e^{- C_1 hK } \Big), 
$$
and the previous result has to be modified accordingly in these cases. 
\end{remark}


\begin{thebibliography}{99}

\bibitem{Adh1}
{\rm S.K. Adhikari},
{\em Numerical solution of the two-dimensional Gross-Pitaevskii equation for trapped interacting atoms}, Phys. Lett. A, 265 (2000) 91--96. 

\bibitem{Adh2}
{\rm S.K. Adhikari},
{\em Numerical solution of spherically symmetric Gross-Pitaevskii equation in two space dimensions}, Phys. Lett. E, 62 (2000) 2937--2944. 

%\bibitem{Ascher}
%{\rm U. M. Ascher, S. Reich},
%{\em The midpoint scheme and variants for Hamiltonian systems:
%advantages and pitfalls,} SIAM J. Sci. Comput. 21, (1999) 1045Ð1065.
%
%\bibitem{Baker05}
%{\rm H.F. Baker},
%{\em Alternants and continuous groups},
%Proc. of London Math. Soc. 3 (1905) 24--47. 
%
%\bibitem{BeGi94}
%{\rm G. Benettin and A. Giorgilli},
%{\em On the Hamiltonian interpolation of near to the identity symplectic mappings with application to symplectic integration algorithms}, J. Statist. Phys. 74 (1994), 1117--1143. 
\bibitem{BP10}
D.~Bambusi and T. Penati \emph{Continuous approximation of breathers in one and two dimensional DNLS lattices }, Nonlinearity 23 (2010), no. 1, 143Ð157.  

\bibitem{bao}
{\rm W. Bao and Q. Du}, 
{\em Computing the ground state solution of Bose-Einstein condensates by a normalized gradient flow}. 
SIAM J. Sci. Comput. 25 (2004) 1674--1697. 

\bibitem{baotang}
{\rm W. Bao and W. Tang},
{\em Ground state solution of trapped interacting Bose-Einstein condensate by directly minimizing the energy functional}, J. Comput. Phys. 187 (2003), 230 -- 254. 

\bibitem{BFG}
{\rm D. Bambusi, E. Faou and B. Gr\'ebert}
{\em Existence and stability of ground states for fully discrete approximations of the nonlinear Schr\"odinger equation}. Numer. Math. 123 (2013) 461--492

\bibitem{CST00}
{\rm M.L. Chiofalo, S. Succi and M.P. Tosi}, 
{\em Ground state of trapped interacting Bose-Einstein condensates by an explicit imaginary time algorithm}, Phys. Rev. E, 62 (2000) 7438--7444.

%\bibitem{BG06}
%{\rm D. Bambusi and B. Gr\'ebert},
%{\em Birkhoff normal form for PDE's with tame modulus}. Duke Math. J.  135  no. 3 (2006), 507Ð-567.
%%
%
%\bibitem{Cano06}
%{\rm B. Cano},
%{\em Conserved quantities of some Hamiltonian wave equations after full discretization}, Numer. Math. 103 (2006) 197--223. 
%
%\bibitem{Cazenave}
%{\rm T. Cazenave},
%{\em Semilinear Schr\"odinger equations}. Courant Lecture Notes in Mathematics, 10. New York University, Courant Institute of Mathematical Sciences, New York; American Mathematical Society, Providence, RI, 2003.
%
%%\bibitem{CD09}
%%%{\rm F. Castella, G. Dujardin},
%%{\em Propagation of Gevrey regularity over long times for the fully discrete Lie-Trotter splitting scheme applied to the linear Schr\"odinger equation}, to appear in M2AN. 
%
%\bibitem{CHL08c}
%{\rm D. Cohen, E. Hairer and C. Lubich},
%{\em Conservation of energy, momentum and actions in numerical discretizations of nonlinear wave equations}, 
%Numer. Math. 110 (2008) 113--143.
%
%\bibitem{CFL28}
%{\rm R. Courant, K. Friedrichs and H. Lewy}, 
%{\em \"Uber die partiellen Differenzengleichungen der mathematischen Physik}, Mathematische Annalen 100 (1928) 32Ð-74.
%
%%\bibitem{Duran00}
%%{\rm A. Dur\'an, J.-M. Sanz-Serna},
%%{\em The numerical integration of relative equilibrium solutions. The nonlinear Schršdinger equation.}
%%IMA J. Numer. Anal. 20 (2000), no. 2, 235--261. 
%
%%\bibitem{Elia06}
%%{\rm H. L. Eliasson, S. B. Kuksin},
%%{\em KAM for non-linear Schroedinger equation},  
%%Preprint  (2006) 

%\bibitem{DF09}
%{\rm A. Debussche and E. Faou},
%{\em Modified energy for split-step methods applied to the linear Schr\"odinger equation}
%SIAM J. Numer. Anal. 47 (2009) 3705--3719.
%
%\bibitem{DF07}
%{\rm G. Dujardin and E. Faou},
%{\em Normal form and long time analysis of splitting schemes for the linear Schr{\"o}dinger equation with small potential.}
%Numerische Mathematik 106, 2 (2007) 223--262 

\bibitem{EB95}
{\rm M. Edwards and K. Burnett}, 
{\em Numerical solution of the nonlinear Schr\"odibnger equation for small samples of trapped neutral atoms}, Phys. Rev. 1, 51 (1995) 1382--1386.
 
%\bibitem{FG2}
%{\rm E. Faou and B. Gr\'ebert}, 
%{\em Resonances in long time integration of semilinear Hamiltonian PDEs}, Preprint 2009: arXiv:0904.1459


%\bibitem{F11}
%{\rm E. Faou},
%{\em Geometric numerical integration and Schr\"odinger equations}. European Math. Soc., 2012. 
%
%\bibitem{FGL1}
%{\rm E. Faou, L. Gauckler and C. Lubich}
%{\em Sobolev stability of plane wave solutions to the cubic nonlinear Schr\"odinger equation on a torus},
%Comm. PDE 38 (2013) 1123-1140
%
%\bibitem{FGL2}
%{\rm E. Faou, L. Gauckler and C. Lubich}
%{\em Plane wave stability of the split-step Fourier method for the nonlinear Schr\"odinger equation}, preprint. 
%
%\bibitem{FG11}
%{\rm E. Faou and B. Gr\'ebert},
%{\em Hamiltonian interpolation of splitting approximations for nonlinear PDE's}. 
%Found. Comput. Math. 11 (2011) 381--415

\bibitem{FGJS}
{\rm J. Froehlich, S. Gustafson, L. Jonsson and I.M. Sigal}
{\em Solitary wave dynamics in an external potential}. Comm. Math. Phys. 250 (2004), 613-642.


\bibitem{Grill87}
{\rm M. Grillakis, H. Shatah and W. Strauss}, 
{\em Stability theory of solitary waves in the presence of symmetry. I.},
J. Funct. Anal. 74 (1987) 160--197. 

\bibitem{Grill90}
{\rm M. Grillakis, H. Shatah and W. Strauss}, 
{\em Stability theory of solitary waves in the presence of symmetry. II.},
J. Funct. Anal. 94 (1990) 308--348. 

\bibitem{hairerlubich}
{\rm E. Hairer and Christian Lubich},
{\em  Energy-diminishing integration of gradient systems}, 
IMA J. Numer. Anal. 34 (2014), 452--461.
%\bibitem{HLW}
%{\rm E. Hairer, C. Lubich and G. Wanner},
%{\em Geometric Numerical Integration. Structure-Preserving Algorithms for Ordinary Differential Equations}. Second Edition. Springer 2006. 
%

%\bibitem{Reich99}
%{\rm S. Reich},
%{\em Backward error analysis for numerical integrators}, SIAM J. Numer. Anal. 36 (1999) 1549--1570. 
%
%

%\bibitem{FGP1}
%{\rm E. Faou, B. Gr\'ebert and E. Paturel},
%{\em Birkhoff normal form for splitting methods applied to semilinear Hamiltonian PDEs. Part I: finite-dimensional discretization.} Numer. Math.  114 (2010) 429--458. 
%
%\bibitem{FGP2}
%{\rm E. Faou, B. Gr\'ebert and E. Paturel},
%{\em Birkhoff normal form for splitting methods applied to semilinear Hamiltonian PDEs. Part II: Abstract splitting.} Numer. Math. 114 (2010) 459--490. 
%
%
%\bibitem{GL08a}
%{\rm L. Gauckler and C. Lubich},
%{\em Nonlinear Schr\"odinger equations and their spectral discretizations over long times},
%Found. Comput. Math. 10 (2010), 141-169.
%
%\bibitem{GL08b}
%{\rm L. Gauckler and C. Lubich},
%{\em Splitting integrators for nonlinear Schr\"odinger equations over long times},
%Found. Comput. Math. 10 (2010), 275--302.

%
%
%\bibitem{Greb07}
%{\rm B. Gr\'ebert},  
%{\em Birkhoff normal form and Hamiltonian PDEs.}
%S\'eminaires et Congr\`es 15 (2007), 1--46
%
%\bibitem{HL97}
%{\rm E. Hairer and C. Lubich},
%{\em The life-span of backward error analysis for numerical integrators}, 
%Numer. Math. 76 (1997) 441--462
%
%\bibitem{CHL08b}
%{\rm  E. Hairer and C. Lubich}, 
%{\em Spectral semi-discretisations of weakly nonlinear wave equations over long times}, 
%Found. Comput. Math. 8 (2008) 319-334.
%%
%\bibitem{HLW}
%{\rm E. Hairer, C. Lubich and G. Wanner},
%{\em Geometric Numerical Integration. Structure-Preserving Algorithms for Ordinary Differential Equations}. Second Edition. Springer 2006. 
%
%%\bibitem{Hans08}
%%{\rm E. Hansen, A. Ostermann},
%{\em Exponential splitting for unbounded operators.}
%Math. Comp. 78  (2009) 1485--1496. 
%
%\bibitem{Haus06}
%{\rm F. Hausdorff},
%{\em Die symbolische Exponentialformel in der Gruppentheorie.}
%Berichte des S\"achsischen Akad. der Wissensch. 58 (1906), 19--48. 
%
%

%\bibitem{JL00}
%{\rm T. Jahnke and C. Lubich}, 
%{\em Error bounds for exponential operator splittings}, BIT 40 (2000) 735-744. 
%
%\bibitem{L08}
%{\rm C. Lubich}, {\em From quantum to classical molecular dynamics: reduced models and numerical analysis}. ÊEuropean Math. Soc., 2008. 
%
%\bibitem{L08a}
%{\rm C. Lubich}, {\em On splitting methods for Schr\"odinger-Poisson and cubic nonlinear Schr\"odinger equations}, Math. Comp. 77 (2008), 2141-2153.


%\bibitem{Hair93}
%{\rm E. Hairer, S. P. N{\o}rsett, G. Wanner },
%{\em Solving Ordinary Differential Equations I. Nonstiff
%                  Problems.},
%Springer Series in Computational Mathematics 8, Berlin, 1993, 2nd edition.
% 
%\bibitem{Mcla02}
%{\rm R. I. McLachlan, G. R. W. Quispel}, 
%{\em Splitting methods}, 
%Acta Numerica 11 (2002), 341-434.
% 
%\bibitem{Oliv04} 
%{\rm M. Oliver, M. West, C. Wulff},
%{\em Approximate momentum conservation for spatial semidiscretizations of semilinear wave equations},
%Numer. Math. 97 (2004), 493--535.
%
%\bibitem{Moser68}
%{\rm J. Moser}, 
%{\em Lectures on Hamiltonian systems}, Mem. Am. Math. Soc. 81 (1968) 1--60.
%
%\bibitem{Reich99}
%{\rm S. Reich},
%{\em Backward error analysis for numerical integrators}, SIAM J. Numer. Anal. 36 (1999) 1549--1570. 
%
%\bibitem{Reic04} 
%{\rm B. Leimkuhler, S. Reich},
%{\em Simulating Hamiltonian dynamics}.  
%Cambridge Monographs on Applied and Computational Mathematics, 14. Cambridge University Press, Cambridge, 2004. 
%        

%\bibitem{Jahn00}
%{\rm T. Jahnke, C. Lubich},
%{\em Error bounds for exponential operator splittings},
% BIT 40 (2000), 735--744.
%
%
%
%\bibitem{Stern}
%{\rm A. Stern, E. Grinspun},
%{\em Implicit-explicit variational integration of highly oscillatory problems}, preprint (2008). 
%
%

\bibitem{Weinstein85}
{\rm M. I. Weinstein}, 
{\em Modulational stability of ground states of nonlinear Schr\"odinger equations}, 
SIAM J. Math. Anal. 16 (1985) 472--491. 

\end{thebibliography}
\end{document}